\documentclass[letterpaper, 10 pt, conference]{ieeeconf}  
\IEEEoverridecommandlockouts                              
\usepackage{amsmath,amssymb,amsfonts}
\usepackage{cuted}
\usepackage{multicol}
\usepackage{flushend}
\usepackage{xcolor}
\usepackage{algorithm2e}
\usepackage{mathtools}
\usepackage{cite}
\usepackage{graphicx}
\usepackage{floatrow}
\mathtoolsset{showonlyrefs}
\usepackage[colorinlistoftodos, textwidth=3cm, shadow]{todonotes}

\PassOptionsToPackage{hyphens}{url}\usepackage{hyperref}
\newtheorem{theorem}{Theorem}[section]
\newtheorem{definition}{Definition}[section]
\newtheorem{assumption}{Assumption}[section]
\newtheorem{problem}{Problem}

\newtheorem{lemma}[theorem]{Lemma}

\newtheorem{remark}{Remark}

\title{{\LARGE \bf Risk-averse controller design against data injection attacks on actuators for uncertain control systems}\\
}
\author{
Sribalaji C. Anand$^{1}$ and Andr\'e M. H. Teixeira$^{2}$
\thanks{*This work is supported by the Swedish Research Council under the grant 2018-04396 and by the Swedish Foundation for Strategic Research.}
\thanks{$^{1}$ Sribalaji C. Anand is with the Department of Electrical Engineering, Uppsala University, PO Box 65, SE-75103, Uppsala, Sweden. {\tt\small sribalaji.anand@angstrom.uu.se}}%
\thanks{$^{2}$ Andr\'e M. H. Teixeira is with the Department of Information Technology, Uppsala University, PO Box 337, SE -75105, Uppsala, Sweden. {\tt\small andre.teixeira@it.uu.se}}%
}
\begin{document}
\maketitle
\thispagestyle{empty}
\pagestyle{empty}
\begin{abstract}

In this paper, we consider the optimal controller design problem against data injection attacks on actuators for an uncertain control system. We consider attacks that aim at maximizing the attack impact while remaining stealthy in the finite horizon. To this end, we use the Conditional Value-at-Risk to characterize the risk associated with the impact of attacks. The worst-case attack impact is characterized using the recently proposed output-to-output $\ell_2$-gain (OOG). We formulate the design problem and observe that it is non-convex and hard to solve. Using the framework of scenario-based optimization and a convex proxy for the OOG, we propose a convex optimization problem that approximately solves the design problem with probabilistic certificates. Finally, we illustrate the results through a numerical example.
\end{abstract}
\section{INTRODUCTION}\label{sec_intro}
Cyber-physical systems (CPSs) represent a large class of networked control systems where the physical world and the digital infrastructure are tightly coupled, such as smart cities, autonomous systems, transportation networks, and Internet Of Things. However, the trend towards increased usage of open-standard communication protocols among control systems has made these systems vulnerable to online cyber-attacks such as Stuxnet \cite{BG1}, Industroyer \cite{cherepanov2017industroyer}, etc. Such cyber-attacks can negatively affect the operation of CPS \cite{BG3}. 

Significant work is done in detecting and mitigating cyber-attacks (see \cite{dibaji2019systems,griffioen2019tutorial} and references therein). For instance, \cite{muller2018risk} designs an optimal controller in the presence of covert attacks. The limitation of \cite{muller2018risk} is, it approximates the risk metric Conditional Value-at-Risk (CVaR) empirically using samples, and it parameterizes the controller as a finite family of Finite Impulse Response (FIR) filters. Although these approximations simplify the problem, the validity of these approximations is not discussed except in the asymptotic case i.e., as the number of samples for empirical approximation and the number of FIR filters tends to infinity. 

The article \cite{murguia2020security} proposes and solves two controller design problems. Firstly, it proposes a convex design problem such that the volume of the reachable set of states by the adversary is minimized. Secondly, it proposes a convex design problem that maximizes the Euclidean distance between the set of states reachable by the adversary and the set of critical states. A similar approach was also adopted in \cite{hashemi2020gain}. However, both of these works do not consider an uncertain system. The works \cite{anand2020joint} and \cite{bopardikar2016h} addresses the issue of jointly designing the controller and detector against false data injection (FDI) attacks. However, they also assume a deterministic system.

This paper addresses some of the existing limitations in the literature, by investigating the optimal controller design problem against FDI attacks on actuators for an uncertain control system. To this end, we adopt the following setup. We consider a discrete-time (DT) linear time-invariant (LTI) process with parametric uncertainty, a static output feedback controller, and an anomaly detector. An adversary with perfect system knowledge injects false data into the actuators. In reality, it is hard for the adversary to have perfect system knowledge, but this assumption helps to study the worst-case. The system operator (or the defender) knows only about the bounds of the uncertainty. Under this setup, we present the following contributions.
\begin{enumerate}
    \item Firstly, we formulate the risk-averse design problem. Here, for a given realization of the uncertainty, we use the output-to-output $\ell_2$-gain (OOG) \cite{teixeira2015strategic} to characterize the worst-case impact. We then use the CVaR to characterize the risk associated with the attack impact. The advantages of using the OOG over the classical $H_{\infty}/H_{\_}$  metrics were demonstrated in \cite{anand2020joint}. We also observe that the design problem corresponds to an untractable infinite non-convex optimization problem.
    \item Secondly, extending the results of \cite{dvijotham2014convex}, we derive an upper bound for the OOG. Using this upper bound, we relax the infinite non-convex design problem into an infinite convex design problem.
    \item Finally, by adopting the scenario-based approach \cite{ramponi2018expected}, we modify the infinite convex optimization problem into its sampled counterpart. We also provide probabilistic guarantees on the infinite design problem based on the number of samples used to formulate the sampled optimization problem and the dimension of the controller. The advantage of using scenario-based approach over other approaches is discussed in \cite{campi2018introduction}.
\end{enumerate}

To the best of the author's knowledge, the problem of risk-sensitive controller design for an uncertain control in the finite horizon against FDI attacks has not been addressed in the literature.

The remainder of this paper is organized as follows. Section \ref{sec_PB} describes the problem background. The design problem is formulated in Section \ref{sec_PF}. The problem is relaxed and convexified in Section \ref{sec_convex}. Section \ref{sec_sample} approximates the problem empirically using the scenario-based approach. We illustrate the results using a numerical example in Section \ref{sec_NE}. Finally, we provide concluding remarks in Section \ref{sec_conclusion}.

\section{Problem background}\label{sec_PB}
In this section, we describe the control system structure and the goal of the adversary. Consider the general description of a finite horizon closed-loop DT LTI system with a process ($\mathcal{P}$) with parametric uncertainty, a static output feedback controller ($\mathcal{C}$) and an anomaly detector ($\mathcal{D}$) as shown in Fig. \ref{System}. The closed-loop system is represented by 
\begin{align}
    \mathcal{P}: & \left\{
                \begin{array}{ll}
                    x_p[k+1] &= A^{\Delta}x_p[k] + B \tilde{u}[k]\\
                    y[k] &= Cx_p[k]\\
                    y_p[k] &= C_Jx_p[k]
                \end{array}
                \right. \label{P}\\
    \mathcal{C}: & \left\{
                 \begin{array}{ll}
                         {u}[k] &= Ky[k]
                \end{array}
                \right. \label{C} \\
    \mathcal{D}: & \left\{
                \begin{array}{ll}
                \hat{x}_p[k+1] &= A\hat{x}_p[k] + Bu[k] +Ly_r[k]\\
                y_r[k] &= y[k] - C\hat{x}_p[k], \; k=0,\dots,N_h-1.
                \end{array}
                \right. \label{D}
\end{align} 

Here $A^{\Delta} \triangleq A + \Delta A(\delta)$ with $A$ representing the nominal system matrix and {\color{black}$\delta \in \Omega$ denoting the probabilistic parameter uncertainty with probability space $(\Omega, \mathcal{D}_a, \mathbf{P})$}. We assume the uncertainty set $\Omega\subset \mathbb{R}^v$ to be closed, bounded, and to include the zero uncertainty yielding $\Delta A(0) = 0$. The state of the process is represented by $x_p[k]$, the output of the process is $y[k]  $, $\tilde{u}[k]$ is the control signal received by the process, $ u[k]$ is the control signal generated by the controller, $y_p[k]$ is the virtual performance output, and $y_r[k]$ is the residue generated by the detector. In this paper, we assume all signals have the same dimension $n_x$. That is, we consider a fully actuated square system. The system is said to have a good performance over the horizon $N_h$, when the energy of the performance output ($||y_p||_{\ell_2, [0,N_h]}^2$) is small. In the closed-loop system described above, we consider that an adversary is injecting false data into the actuators. An attack is said to be detected when the energy of the detection output ($||y_r||_{\ell_2, [0,N_h]}^2$) is higher than a predefined threshold (say $\epsilon_r$). We assume that the detection threshold and the detector $L$ is designed to be robust against all uncertainties.

Given this setup, we now discuss the resources the adversary has access to. For clarity, we establish the following:
\begin{assumption}\label{ass_contr}
$(A^\Delta,B)$ is controllable $\forall \delta\in\Omega$. $\hfill\triangleleft$
\end{assumption}
\begin{assumption}
Matrices $B$ and $C$ are invertible. $\hfill\triangleleft$
\end{assumption}
\begin{figure}
    \centering
    \includegraphics[width=8.4cm]{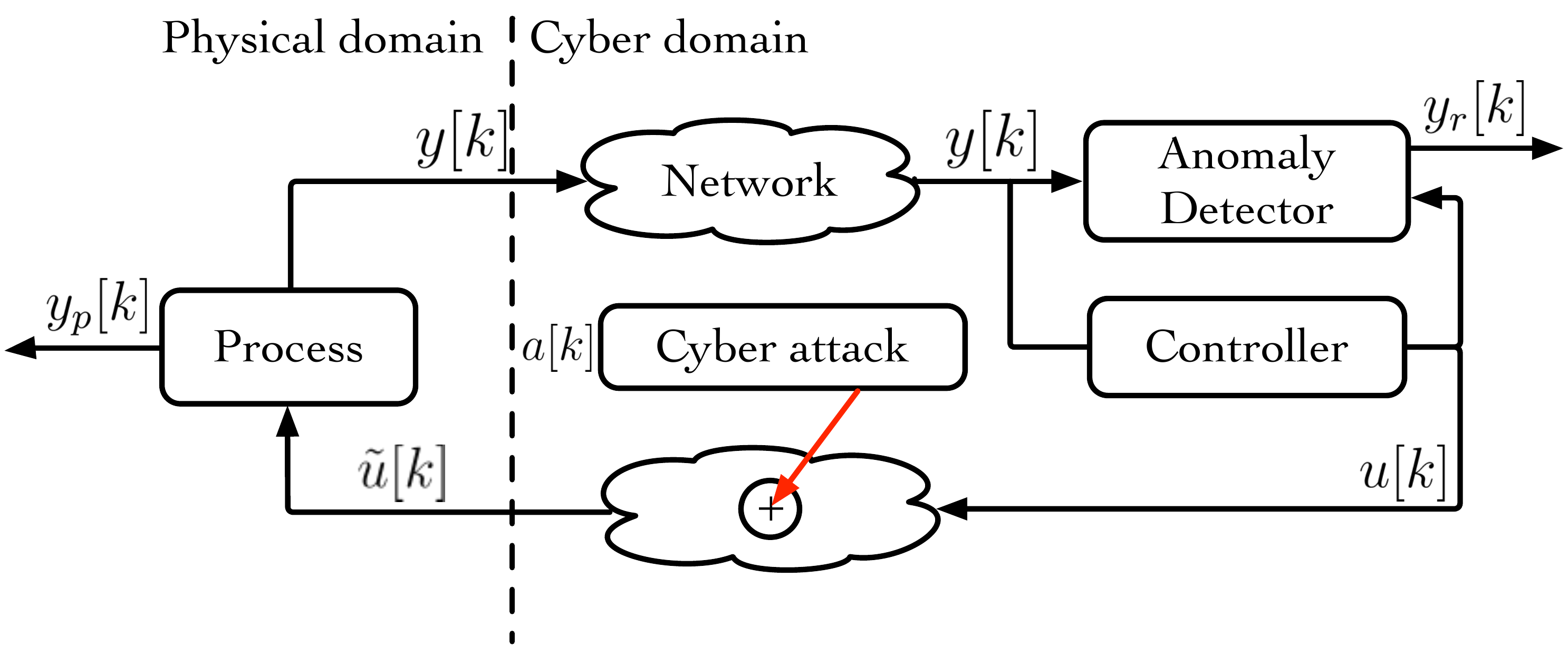}
    \caption{Control system under data injection attack on actuators}
    \label{System}
\end{figure}
\subsection{Disruption and disclosure resources}\label{disclosure:sec} 
The adversary can access the control channels and inject data. This is represented by $\tilde{u}[k]={u}[k]+a[k],$ where $a[k] \in \mathbb{R}^{n_x}$ is the data injected by the adversary. The adversary cannot access the sensor channels. The adversary does not have access to any disclosure (eavesdropping) resources.
\subsection{System knowledge} 
We assume that, at design time, the defender knows the bounds of the set $\Omega$ and that the system matrix $A^{\Delta}$ is known only up to the nominal system matrix $A$. Next, at operation time, we assume that the adversary has full system knowledge. That is, the adversary knows the system matrix $A^{\Delta}$ without any uncertainties. In reality, it is hard for the adversary to know the system matrices, but this assumption helps to study the worst case.

The system knowledge is used by the adversary to calculate the optimal data injection attacks. Defining $e[k] \triangleq x_p[k]-\hat{x}_p[k]$ and ${x}[k] \triangleq [ x_p[k]^T \; e[k]^T]^T$, the closed-loop system under attack with the performance output and detection output as system outputs becomes
\begin{equation}\label{system:CL}
    \mathcal{P}_{cl}:\left\{
                \begin{array}{ll}
                            {x}[k+1] &= {A}_{cl}^{\Delta}{x}[k] + {B}_{cl}a[k]\\
                            y_p[k] &= {C}_p{x}[k]\\
                            y_r[k] &= {C}_r{x}[k],\\
                \end{array}
                \right.
\end{equation}
\begin{align*}
\text{where} \;\; &{A}_{cl}^{\Delta} \triangleq \begin{bmatrix}
    A^{\Delta}+BKC & 0\\
    \Delta A & A-LC
    \end{bmatrix}, & {B}_{cl} \triangleq \begin{bmatrix}
    B \\ B 
    \end{bmatrix},\\
    &{C}_p \triangleq \begin{bmatrix}
    C_J & 0
    \end{bmatrix}, & {C}_r \triangleq \begin{bmatrix}
    0 & C
    \end{bmatrix}.
\end{align*}
\subsection{Attack goals and constraints}
Given the resources the adversary has access to, the adversary aims at disrupting the system's behavior whilst remaining stealthy. The system disruption is evaluated by the increase in energy of the performance output, and the attack signal is deemed to be stealthy when the energy of the detection output is less than $\epsilon_r$. Next, we discuss the optimal attack policy of the adversary and the design problem of the defender when the system is deterministic.
\subsection{Design for a deterministic system}
From the previous discussions, it can be understood that the goal of the adversary is to maximize the energy of the performance output whilst remaining stealthy. When the system is deterministic
($\Omega=\{0\}$), the attack policy of the adversary can be formulated as
\begin{equation}\label{opti_oog_primal}
\begin{aligned}
q(K,0)  \triangleq \sup_{a\in \ell_{2e}}\quad  & \Vert y_p(K,0) \Vert_{\ell_2}^2 \\
\textrm{s.t.} \quad & \Vert y_r(K,0) \Vert_{\ell_2}^2 \leq \epsilon_r,\; x(K,0)[0]=0,
\end{aligned}
\end{equation}
where the subscript $[0,N_h]$ is dropped for clarity. In \eqref{opti_oog_primal}, $q(K,0)$ is the disruption caused by the attack signal on the nominal system, $y_p(K,0)$ and $y_r(K,0)$ are the performance output and the detection output under the given controller $K \in \mathbb{R}^{n_x \times n_x}$, and $N_h$ is the horizon length. In \eqref{opti_oog_primal}, the constraint $x(K,0)[0]=0$ is introduced since the system is at equilibrium before the attack commences.
\begin{assumption}
The system \eqref{system:CL} is at equilibrium before the attack commences. $\hfill\triangleleft$
\end{assumption}

The aim of the defender then is to design a controller $K$ such that the disruption caused by the adversary ($q(K,{\color{black} 0})$) is minimized. To this end, the design problem can be formulated as
\begin{equation}\label{dummy1}
{\color{black}K^* = \arg \inf_{K} q(K,0)}
\end{equation}

\color{black}{The design problem \eqref{dummy1} is optimal only when \eqref{system:CL} is deterministic}\color{black}{}. By extending \eqref{dummy1}, we formulate the design problem when the system is uncertain in the next section.

\section{Problem Formulation}\label{sec_PF}
Consider the data injection attack scenario where the parametric uncertainty $\delta \in \Omega$ of the system is known to the adversary but not to the defender. The defender knows only about the probabilistic description of the set $\Omega$. In reality, it is hard for the adversary to know the system matrices, but this assumption helps to study the worst case. Under this setup, the adversary can cause high disruption by remaining stealthy as it will be able to inject attacks by solving
\begin{equation}\label{eqq4}
\begin{aligned}
q(K,\delta)  \triangleq \sup_{a\in \ell_{2e}} \quad & \Vert y_p(K,\delta) \Vert_{\ell_2}^2 \\
\textrm{s.t.} \quad & \Vert y_r(K,\delta) \Vert_{\ell_2}^2 \leq \epsilon_r,\;x(K,\delta)[0]=0,
\end{aligned}
\end{equation}
where $y_p({K,\delta})$ and $y_r({K,\delta})$ are the performance and detection output corresponding to the controller $K$ and uncertainty $\delta$. Since the defender does not know the system completely, $q(K,\delta)$ becomes a random variable. Thus, from the defenders point of view, the best option is to choose a feedback policy $K$, such that the risk corresponding to the impact random variable $q(K,\delta)$ is minimized. This design problem can be formulated as \textit{Problem \ref{problem_0}}.
\begin{problem}\label{problem_0}
Find an optimal feedback controller $K^*$ s.t:
\begin{equation}
K^* \triangleq \arg \inf_{K} \mathcal{R}_{\Omega}(q(K,\delta)),
\end{equation}
where $\mathcal{R}_{\Omega}$ is a risk metric chosen by the defender. The subscript $\Omega$ denotes that the risk acts on the uncertainty whose probabilistic description is known to the defender. 
$\hfill\triangleleft$
\end{problem}

\textit{Problem \ref{problem_0}} searches for a controller $K$ such that the risk is minimized. Let us consider the setup where the defender evaluates the risk based on the risk metric CVaR. CVaR is used in the research community due to its numerous advantages \cite{rockafellar2000optimization} and is defined in \textit{Definition \ref{def_Var}}.
\begin{definition}[CVaR \cite{ramponi2018expected}]\label{def_Var}
Given a random variable $X$ and $\alpha \in (0,1)$, the CVaR is defined as \footnote{\color{black}{This Definition assumes the distribution of $X$ has no point masses. For general definitions of CVaR see \cite{rockafellar2002conditional}.}}
\[
\text{CVaR}_{\alpha}(X) = \mathbb{E}\{X|X > \text{VaR}_{\alpha}(X)\},
\]
\[\text{where} \quad \quad 
\text{VaR}_{\alpha}(X) = \inf\{x|\mathbb{P}[X \geq x] \leq \alpha\}.
\]
$\text{CVaR}_{\alpha}(X) = \beta$ implies that $X \leq \beta$ at least $\alpha \times 100\%$ of the time on average. $\hfill\triangleleft$
\end{definition}

In our setting, the defender is interested in determining the controller such that the $\text{CVaR}_{\alpha}$ (given $\alpha$) of the impact random variable ($q(K,\delta)$) is minimized. To this end, \textit{Problem 1} can be reformulated as
\begin{equation}\label{problem_1}
K^* = \arg\inf_{K} \mathbb{E}_{\Omega}\{q(K,\delta)|q(K,\delta) > \text{VaR}_{\alpha}(q(K,\delta))\}.
\end{equation}

There are two difficulties in solving \eqref{problem_1}. Firstly, \eqref{eqq4} is non-convex for any given $\delta$. Secondly, since the operator $\mathbb{E}$ operates over the continuous space $\Omega$, the optimization problem \eqref{problem_1} is computationally intensive and in general NP-hard. To this end, in \textit{Section \ref{sec_convex}}, we determine a convex approximation for \eqref{eqq4}. We then use this approximation, to recast \eqref{problem_1} as a convex optimization problem. In \textit{Section \ref{sec_sample}}, we provide a method to approximate the expectation operator.

\section{Design problem formulation using a convex Impact proxy}\label{sec_convex}
In this section, we consider the function $q(K,\delta_j)$ for a given uncertainty $\delta_j \in \Omega$ and prove that it has an upper bound. {\color{black}We then show that the term of the upper bound that is dependent on the controller (say $\bar{q}(\cdot)$) is convex in $K$. The main objective of performing this step is that, once we determine the term $\bar{q}(\cdot)$, it can be used in \eqref{problem_1} instead of $q(K,\delta_j)$ to formulate a relaxed convex design problem. To this end, we will refer to $\bar{q}(\cdot)$ as \textit{Impact proxy} in the reminder of the paper.}

To derive the upper bound, we begin by defining the vectors
$\textbf{a}_j \triangleq \begin{bmatrix}
a_j[0]^T,\dots,a_j[N_h-1]^T
\end{bmatrix}^T$, $\textbf{x}_{p,j} \triangleq \begin{bmatrix}
x_{p,j}[1]^T,\dots,x_{p,j}[N_h]^T
\end{bmatrix}^T$, $\textbf{e}_{j} \triangleq \begin{bmatrix}
e_{j}[1]^T,\dots,e_{j}[N_h]^T
\end{bmatrix}^T$, $\textbf{y}_{p,j} \triangleq \begin{bmatrix}
y_{p,j}[1]^T,\dots,y_{p,j}[N_h]^T
\end{bmatrix}^T$, and $\textbf{y}_{r,j} \triangleq \begin{bmatrix}
y_{r,j}[1]^T,\dots,y_{r,j}[N_h]^T
\end{bmatrix}^T$. Here $\textbf{a}_{j}, \textbf{x}_{p,j}, \textbf{y}_{p,j}$ and $\textbf{y}_{r,j}$ are the stacked attack vector, system state, performance output vector and the detection output vectors  corresponding to the uncertainty $\delta_j$ respectively. Let us define the matrices $F_{xa}(K,\delta_j), F_{ea}(\delta_j), F_{ex}(\delta_j) \in \mathbb{R}^{n_xN_h \times n_xN_h}$, such that 
\begin{align}
    \textbf{x}_{p,j} &= F_{xa}(K,\delta_j)\textbf{a}_j,\; \textbf{e}_j = F_{ea}(\delta_j)\textbf{a}_j  + F_{ex}(\delta_j)\textbf{x}_{p,j}, \\
    \textbf{y}_{p,j}&= F_{p}(K,\delta_j)\textbf{a}_j, \;\; \textbf{y}_{r,j} = F_{r}(K,\delta_j)\textbf{a}_j,\\
    F_{p}&(K,\delta_j) \triangleq (I_{N_h} \otimes C_J) F_{xa}(K,\delta_j),\\
    F_{r}&(K,\delta_j) \triangleq (I_{N_h} \otimes C) (F_{ea}(\delta_j)+F_{ex}(\delta_j)F_{xa}(K,\delta_j)).
\end{align}

Under the uncertainty $\delta_j$, let us represent the system matrix of \eqref{P} by $A_j$. Then $F_{\upsilon a}(K,\delta_j), \upsilon = \{x,e\}$ is given by
\[
\begin{bmatrix}
B & 0 & \dots & 0\\
{A}_{\upsilon,j}B& B  & \dots & 0\\
\vdots & \vdots & \ddots & \vdots\\
{A}_{\upsilon,j}^{N_h-1}B & {A}_{\upsilon,j}^{N_h-2}B  & \dots & B
\end{bmatrix},
\]
where ${A}_{x,j} \triangleq A_j+BKC$ and $A_{e,j} \triangleq A_j-LC$. Similarly $F_{ex}(\delta_j)$ is given by
\begin{align}
&\begin{bmatrix}
0 & 0 & \dots & 0\\
\Delta A & 0  & \dots & 0\\
\vdots & \vdots & \ddots & \vdots\\
{A}_{e,j}^{N_h-2}\Delta A & {A}_{e,j}^{N_h-3}B \Delta A  & \dots & 0
\end{bmatrix}
\end{align}
Under these definitions, $q(K,\delta_j)$ can be obtained by the non-convex optimization problem 
\begin{equation}\label{e1}
\begin{aligned}
q(K,\delta_j)\triangleq\sup_{\textbf{a}_j}\;&\Vert F_p(K,\delta_j) \textbf{a}_j{\color{black}\Vert_{2}^2}\\
\textrm{s.t.}\;&\Vert F_r(K,\delta_j)\textbf{a}_j{\color{black}\Vert_{2}^2}\leq\epsilon_r, x(K,\delta_j)[0]=0.
\end{aligned}
\end{equation}
Next, we derive the upper bound of $q(K,\delta_j)$ in \textit{Lemma \ref{lem2}}.
\begin{lemma}\label{lem2}
Let $\delta_j \in \Omega$ and $\kappa \triangleq F_p(\cdot)F_r^{-1}(\cdot)$ (Here the arguments of $F_p$ and $F_r$ are dropped for clarity). Let the matrix $B$ be invertible. Then, it holds that
\begin{equation}\label{f1}
q(K,\delta_j) \leq \mu \bar{q}(K,\delta_j)^f,\;\bar{q}(K,\delta_j) \triangleq\eta||K||_F^2+ \sum_{i=2}^{n_xN_h}\sigma_i(\kappa^{-1}),
\end{equation}
where $f \triangleq n_xN_h-1, \mu$ is a term independent of $K$ and {\color{black} $\eta$ is a positive scalar weight on the regularization term.}
\end{lemma}
\begin{proof}
See Appendix.
\end{proof}

In \textit{Lemma \ref{lem2}}, we formulated an upper bound for ${q}(K,\delta_j)$. However, only the term $\bar{q}(K,\delta_j)$ of the bound is dependent on the variable $K$. Moreover,  since $\left(\bar{q}(K,\delta_j)\right)^{n_xN_h-1}$ is a monotonically increasing function on $\bar{q}(K,\delta_j)>0$, $\bar{q}(K,\delta_j)$ can be replaced as the term to be optimized. Next, we show that $\bar{q}(K,\delta_j)$ is {\color{black}strongly} convex in the design variable $K$. 
\begin{theorem}\label{lem3}
For any given $\delta_j \in \Omega$, the function $\bar{q}(K,\delta_j)$ is {\color{black}strongly} convex in the design variable $K$.
\end{theorem}
\begin{proof}
See Appendix.
\end{proof}

We have shown in this section that the term $\bar{q}(K,\delta_j)$ can be used as the convex proxy objective function for the attack impact ${q}(K,\delta_j)$. 
That is, we can recast \eqref{problem_1} as
\begin{equation}\label{problem_2}
K^* = \arg\inf_{K} \mathbb{E}\{\bar{q}(K,\delta)|\bar{q}(K,\delta) > \text{VaR}_{\alpha}(\bar{q}(K,\delta))\}.
\end{equation}

Although \eqref{problem_2} is convex, it is computationally intensive due to the expectation operator. In \textit{Section \ref{sec_sample}}, we discuss a method to approximate the expectation operator.

\begin{remark}
{\color{black} We use Definition \ref{def_Var} to formulate \eqref{problem_2}. Thus \eqref{problem_2} implicitly assumes that the the distribution of $\bar{q}$ has no point masses. However, verifying this conditions is beyond the scope of this paper and is left for future work.} 
\end{remark}


\section{Empirical risk using scenario based approach}\label{sec_sample}
The optimization problem \eqref{problem_2} is computationally intensive since it involves an expectation operator which acts on a continuum of uncertainties $\Omega$. In this section, we provide a method to approximate the expectation operator using the scenario-based approach \cite{ramponi2018expected}. To begin with, let us establish the following:
\begin{assumption}\label{ass_convex}
For any $\delta$, $\bar{q}(\cdot,\delta)$ is a convex function in the design variable $K$. $\hfill\triangleleft$
\end{assumption}

We have shown in \textit{Theorem \ref{lem3}} that \textit{Assumption \ref{ass_convex}} is satisfied. Next, we approximate the expectation operator in \eqref{problem_2} empirically. To do this, let us begin by sampling the uncertainty set $\Omega$ with $N$ samples. Let us consider that $(\delta_1,\dots,\delta_N)$ is a collection of $N$ independent realizations from $\Omega$ and let $\Omega_N \triangleq \{1,\dots,N\}$. Then for any given $K$, and $i \in \Omega_N$, we denote by $\bar{q}_i(K)$, the value attained by $\bar{q}(K,\delta_i)$, and we denote by $\bar{q}_{(i)}(K)$, the $N-i+1^{th}$ order statistic. That is $\bar{q}_{(1)}(K) \geq \bar{q}_{(2)}(K) \geq \dots \geq \bar{q}_{(N)}(K)$. Now we present the first result of the section. 
\begin{lemma}
Let the dimension of the design variable $K$ be $d = n_x^2$. Let $N \geq d$ and $m \triangleq \lceil N(1-\alpha) \rceil$. Then, under \textit{Assumption \ref{ass_convex}}, the solution to \eqref{problem_2} can be obtained empirically by solving the convex optimization problem 
\begin{equation}\label{cvar_emp}
K^{*} = \arg \inf_K\frac{1}{m}\sum_{i=1}^m\bar{q}_{(i)}(K).
\end{equation}
\end{lemma}
\begin{proof} It was shown that \eqref{cvar_emp} is the empirical version of \eqref{problem_2} in \cite{ramponi2018expected} (See equations $(3)$ and $(5)$ in \cite{ramponi2018expected}).
\end{proof}

The design problem \eqref{cvar_emp} is formulated using the empirical formulation of the risk metric CVaR. If we could order the functions $ \bar{q}_{(i)}(K)$, then the problem of obtaining the optimal controller would be simple. However, $\bar{q}_{(i)}(K)$ is a function of the optimization variable $K$. Hence it is not possible to explicitly define the order without knowing $K$ beforehand. Alternatively, the design problem \eqref{cvar_emp} can be modified such that the controller is optimal to any possible ordering of the functions. To this end, \eqref{cvar_emp} can be recast as
\begin{equation}\label{sample_final}
\begin{aligned}
\arg \inf_{K,y} \quad & y\\
\textrm{s.t.} \quad & \frac{1}{m}\sum_{j=1}^m\bar{q}_{i_j}(K) \leq y,\\
& \forall \text{choice of $m$ indices} \{i_1,\dots,i_m\} \subseteq \Omega_N.
\end{aligned}
\end{equation}

Since $\{i_1,\dots,i_m\}$ is any subset of $\Omega_N$ with cardinality $m$, the optimization problem \eqref{sample_final} has $ {N \choose m}$ constraints. To summarize, \eqref{sample_final} is the empirical reformulation of \eqref{problem_2}.

Using the solution obtained from \eqref{sample_final}, we provide probabilistic {\color{black}out-of-sample} certificate on the random variable $\bar{q}(K,\delta),\delta \in \Omega$. To this end, let us represent the optimal controller of \eqref{sample_final} as $K_N^*$, and let $N\geq m+d$. Let us define {\color{black}the out-of-sample}  \textit{Probability of Shortfall (PS)} as follows:
\begin{definition}[Probability of Shortfall]
PS is defined as
\begin{equation}
    PS(K_N^*) \triangleq \mathbf{P}\{\delta \in \Omega \;|\; \bar{q}(K_N^*,\delta) \geq \bar{q}_{(m+d)}(K_N^*)\}.\tag*{$\triangleleft$} 
\end{equation}
\end{definition}
By ensuring that the $PS(K_N^*)$ is small (say $\epsilon$), one can ensure that the impact proxy under the optimal controller for any {\color{black}new uncertainty} drawn from the set $\Omega$, $\bar{q}(K_N^*,\delta)$, exceeds a predefined valued $\bar{q}_{(m+d)}(K_N^*)$ (also called as the {\color{black}\textit{shortfall threshold}}) with a small probability $\epsilon$. Now we are ready to present the main result in \textit{Theorem \ref{Thm_PS}}.
\begin{theorem}\label{Thm_PS}
{\color{black}It holds that $\mathbf{P}^N\{PS(K_N^*) \leq \epsilon\} \triangleq$
 \begin{equation}
\int_0^{\epsilon}\frac{\Gamma(N+1)}{\Gamma(m+d)\Gamma(N+1-m-d)}p^{m+d-1}(1-p)^{N-m-d}dp,
\end{equation}
where $\Gamma$ is Euler's Gamma function and $\epsilon \in (0,\;1).$}
\end{theorem}
\begin{proof}
See Appendix. 
\end{proof}

\textit{Theorem \ref{Thm_PS}} provides posteriori results on the confidence with which $PS(K_N^*)$ is below a small threshold $\epsilon$. In other words, \textit{Theorem \ref{Thm_PS}} states that, the confidence of the $PS(K_N^*)$ can be evaluated by knowing the dimension of the decision variable ($d$), the number of samples ($N$), and $m$. 

To recall, in this section we proposed an empirical version of \eqref{problem_2} in \eqref{sample_final}. We also provided probabilistic guarantees on the {\color{black}out-of-sample} PS. We conclude this section by providing \textbf{Algorithm \ref{algo1}} which depicts the outline for solving \textit{Problem 1} approximately. In the next section, we depict the efficacy of the proposed algorithm using a numerical example. 
\begin{algorithm}
\SetAlgoLined
\KwResult{$K_N^*,PS$}
\textbf{Initialization}: $N_h, k, d, \Omega, \epsilon$\;
\begin{enumerate}
    \item Choose $N$ such that $N \geq m+d$.
    \item Extract $N$ independent realizations from $\Omega$. 
    \item \color{black}{Build the matrix $\kappa^{-1}, \forall i \in \Omega_N$}\color{black}{}.
    \item Solve the convex optimization problem \eqref{sample_final}.
    \item Given $\epsilon$, evaluate the confidence using \textit{Theorem \ref{Thm_PS}}.
    \item[]
 \end{enumerate}
 \caption{\color{black}{Risk averse design algorithm}}
 \label{algo1}
\end{algorithm}
\section{Numerical example}\label{sec_NE}
In this section, the efficacy of the design algorithm is depicted through a numerical example. Consider the system of the form \eqref{system:CL} where $C_J=C=I_3$, $A, B$ and $L$ are
\[\begin{bmatrix}
2 & 0 & 1\\
1 & a & 0\\
0 & 1 & b
\end{bmatrix}, \begin{bmatrix}
1 & 1 & 0 \\ 0 & 0.3 & 1\\ 0 & 0 & 1
\end{bmatrix}, \text{and} \;\begin{bmatrix}
1.95 &      0  &  1\\
1  &  0.36  &  1\\
0  &  1  & -0.87
\end{bmatrix},\] respectively. Here $a \in [0.5\; 1.5]$ and $b \in [-0.5\; 0.5]$ are uncertain parameters. The observer gain $L$ is designed using pole placement method.
\begin{table}
\centering
\begin{tabular}{||c | c | c ||}
 \hline
    \;\; & \color{black}{CVaR$_{\alpha}(\bar{q}(K,\delta)+0.1||K||_{F}^2)$} & \color{black}{$\bar{q}_{(m+d)}(K_N^*)$} \\
 \hline
 \color{black}{$m=1$} &  \color{black}{20.7160} & \color{black}{16.9069} \\
 \hline
 \color{black}{$m=2$} & \color{black}{20.6436} & \color{black}{16.8910}\\
 \hline
\end{tabular}
\caption{\color{black}{Risk and Shortfall threshold }}
\label{param}
\end{table}

\begin{figure}
    \centering
    \includegraphics[width=8.4cm]{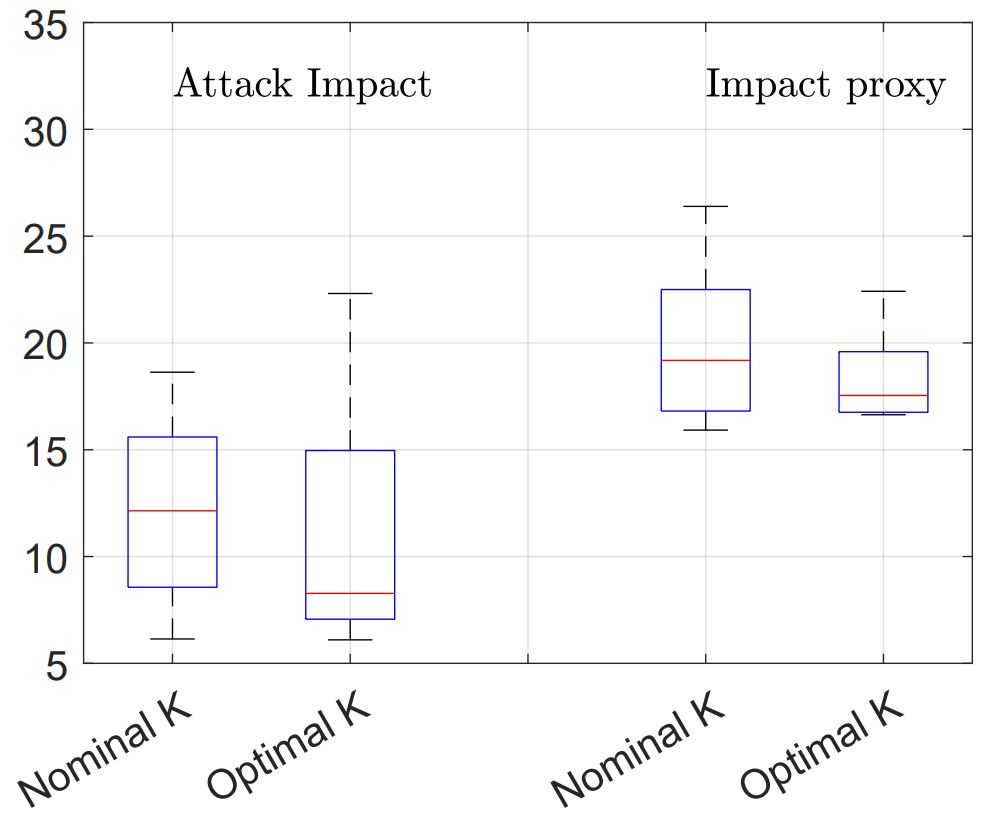}
    \caption{Evaluation of controller performance: The plot depicts the distribution of the attack impact $q(\cdot)$ and the impact proxy $\bar{q}(\cdot)$ for $100$ different uncertainties under the nominal and the optimal controller. The optimal controller is obtained from \textbf{Algorithm \ref{algo1}} by optimizing CVaR$_{0.8}(\bar{q}(\cdot))$ with  parameters $N_h=5, N=11$, and $m=2$. The nominal controller is obtained by optimizing $\bar{q}(K,0)$ for the nominal system. On each box, the central mark indicates the median, and the bottom and top edges of the box indicate the $25^{th}$ and $75^{th}$ percentiles, respectively. The whiskers extend to the most extreme data points.}
    \label{perf_fig}
\end{figure}

Let $N=11,\epsilon_r=1$, and $N_h=5$. The risk and the corresponding shortfall threshold $\bar{q}_{(m+d)}(K_N^*)$ obtained from \textbf{Algorithm \ref{algo1}} for different values of $m$ is shown in TABLE \ref{param}. The corresponding confidence that the the out-of-sample PS is less than or equal to $\epsilon$ can be evaluated from the integral in \textit{Theorem \ref{Thm_PS}}. In reality, the confidence can be made higher by increasing $N$ or decreasing $m$.

Next, we depict the efficiency of the proposed design framework using Fig. \ref{perf_fig}. It can be seen that: (a) compared to the nominal controller, the controller obtained from \textbf{Algorithm 1} lowers the impact proxy $\bar{q}(\cdot)$); (b) The value of attack impact ($q(\cdot)$) is shifted towards a lower median.

\section{Conclusion}\label{sec_conclusion}
The problem of risk-optimal controller design against FDI attacks on actuators of an uncertain control system was studied. We considered an attacker with perfect system knowledge that maximizes the system disruption whilst remaining stealthy. We quantified the system disruption using the OOG. We formulate a design problem where the defender aims at designing a controller such that its CVaR is minimized. We proposed a convex optimization problem that approximately solves the design problem with probabilistic certificates. Finally, we illustrated the results through a numerical example.

\appendix
\renewcommand{\thetheorem}{A.\arabic{theorem}}
\setcounter{theorem}{0}

\section*{Proof of \textit{Lemma \ref{lem2}}}
Before presenting the proof, we provide an intermediate result which helps to construct the proof. 
\begin{lemma}[Inequality of arithmetic and geometric means]\label{lem1}
Given $N$ real numbers $x_1,\dots,x_N$, it holds that $\prod_{i=1}^{N} x_i \leq \big( \frac{\sum_{i=1}^{N}x_i}{N} \big)^{N}$
\end{lemma}
\begin{proof}[Proof of \textit{Lemma \ref{lem2}}]
The convex dual problem of \eqref{e1} can be formulated as \eqref{e2}. Furthermore, it was shown in \cite[\textit{Theorem 3.1}]{cdc2021} that the duality gap is zero.
\begin{equation}\label{e2}
\inf \{ \color{black}{\epsilon_r}\color{black}{}\gamma \;| \;F_p^T(\cdot)F_p(\cdot) - \gamma F_r^T(\cdot)F_r(\cdot) \preceq 0.\}
\end{equation}
Pre-multiplying the constraint of \eqref{e2} by $F_r^{-T}(\delta_j)$ and post-multiplying by  $F_r^{-1}(\delta_j)$, \eqref{e2} can be rewritten as
\begin{equation}\label{e3}
\color{black}{\epsilon_r}\color{black}{}\inf\{\gamma\;|\; \kappa^{T}\kappa \preceq \gamma I.\}
\end{equation}
{\color{black}Since $\epsilon_r$ is a pre-defined constant, it can be moved outside the optimization problem \eqref{e3}.} Using the definition of singular values, \eqref{e3} can be re-written as $\color{black}{\epsilon_r\bar{\sigma}(\kappa)}\color{black}{}$. Thus we have shown that $q(K,\delta_j)= \color{black}{\epsilon_r\bar{\sigma}(\kappa)}\color{black}{}$. Next, we prove that $\bar{\sigma}(\kappa) \leq \mu \bar{q}(K,\delta_j)^{n_xN_h-1}$. To this end, we can show that the matrix $F_{r}(\cdot)$ is a block lower triangular with the element $CB$ in the leading diagonal. Then $F_r(\delta_j)^{-1}$ will be $\begin{bmatrix}
B^{-1}C^{-1} & \dots & 0\\
\vdots & \ddots & \vdots\\
 * & \dots & B^{-1}C^{-1}
\end{bmatrix},$ where $*$ represents that its value is unimportant for now. Then, matrix $\kappa$ will be of the form 
$\begin{bmatrix}
C_JC^{-1}  & \dots & 0\\
\vdots & \ddots & \vdots\\
* & \dots & C_JC^{-1}
\end{bmatrix}.$ Since $\kappa$ is block lower triangular, its determinant is the product of determinant of diagonal blocks \cite[Proof of \textit{Theorem A.1}]{dvijotham2014convex}. Thus $\det(\kappa) = |\det(C_JC^{-1})^{N_h}| \triangleq c$. Since the product of singular values of a matrix is equal to its determinant, we get
\begin{equation}\label{f5}
  \bar{\sigma}(\kappa) = \frac{c}{\prod_{i=1}^{n_xN_h-1}\sigma_i(\kappa)} \stackrel{1}{=}
 c\prod_{i=2}^{n_xN_h}\sigma_i(\kappa^{-1}).  
\end{equation}
The equality $1$ in \eqref{f5} follows since the singular values of $\kappa$ and $\kappa^{-1}$ are reciprocals of each other. By using the result of \textit{Lemma \ref{lem1}}, the term $\epsilon_r\bar{\sigma}(\kappa)$ can be bounded as
\begin{align}
   \epsilon_r\bar{\sigma}(\kappa) &\leq \epsilon_r\vert \det(C_JC)^{N_h}\vert \Big(\frac{\sum_{i=2}^{n_xN_h}\sigma_i(\kappa^{-1})}{n_xN_h-1}\Big)^{n_xN_h-1}\\
    &= \underbrace{\frac{\epsilon_r\vert \det(C_JC)^{N_h}\vert }{(n_xN_h-1)^{n_xN_h-1}}}_{\mu} \Big(\sum_{i=2}^{n_xN_h}\sigma_i(\kappa^{-1})\Big)^{n_xN_h-1}\\
    &\leq \mu \Big(\underbrace{\sum_{i=2}^{n_xN_h}\sigma_i(\kappa^{-1}) + \eta||K||_F^2 }_{\bar{q}(K,\delta_j) }\Big)^{n_xN_h-1} \label{f6}
\end{align}
where the last inequality follows since $\mu$ and $\eta ||K||_F^2$ are non-negative terms. In \eqref{f6}, only the term $\bar{q}(K,\delta_j)$ is dependent on the design variable $K$. This concludes the proof.\end{proof}
\section*{Proof of \textit{Theorem \ref{lem3}}}
\begin{proof}
 Since $\kappa = F_p(\cdot)F_r(\cdot)^{-1}$, it follows that $\kappa^{-1}=$
\begin{equation}
    F_r(K,\delta_j)F_{xa}(K,\delta_j)^{-1}(I_{N_h} \otimes C_J)^{-1}
\end{equation}
\begin{equation}
=(I_{N_h} \otimes C)(F_{ea}(\delta_j)F_{xa}(K,\delta_j)^{-1}+F_{ex}(\delta_j))(I_{N_h} \otimes C_J)^{-1}.
\end{equation}
Here all the matrices are independent of the design variable $K$ except $F_{xa}(K,\delta_j)^{-1}$. It can be verified by  matrix multiplication that $F_{xa}(K,\delta_j)^{-1}$ is of the form
\begin{equation}\label{kinverse}
\begin{bmatrix}
B^{-1} & 0 & 0 & \dots & 0\\
-B^{-1}{A}_{x,j} & B^{-1} & 0 & \dots & 0\\
0 & -B^{-1}{A}_{x,j} & B^{-1} & \dots & 0\\
\vdots & \vdots & \ddots & \vdots\\
0 & 0 & 0 & \dots & B^{-1}
\end{bmatrix}
\end{equation}
The matrix ${A}_{x,j}$ is affine in $K$ and thus the same holds for $\kappa^{-1}$. It can also be shown that the sum of singular values of a matrix, $X \to \sum \sigma(X)$, is convex~\cite{subramani1993sums}. THus we have proven that the term $\sum \sigma(\kappa^{-1})$ is convex in $K$. Since the regularization term $||K||_F^2$ is strictly convex in $K$, we then conclude that $\bar{q}_{i_j}(K,\delta)+\eta||K||_F^2$ is a strictly convex function in $K$ which concludes the proof.
\end{proof}
\section*{Proof of \textit{Theorem \ref{Thm_PS}}}
{\color{black}
Before presenting the proof, we provide an intermediate result \textit{Theorem \ref{thm_app}}: which follows from applying \cite[Theorem 3.1]{ramponi2018expected} to our problem setup. 
\begin{theorem}\label{thm_app}
Let us suppose that $(a)$ a solution to \eqref{sample_final} exists and it is unique almost surely, and $(b)$ for a sample $(\delta_1,\dots,\delta_N)$ of independent realizations from $\Omega$, the event that \{ $\exists\;K | \bar{q}_{i}(K), \forall\; i \in \Omega_N$ has the same value\} has zero probability. Then $\mathbf{P}^N\{PS(K_N^*) \leq \epsilon\} $ has a Beta($m+d,N+1-m-d$) distribution.
\end{theorem}

That is, \textit{Theorem \ref{thm_app}} states that the result of \textit{Theorem \ref{Thm_PS}} follows if $(a)$ and $(b)$ hold: which we prove next. \\
\begin{proof}
\textbf{Proof that $(a)$ holds:} We intend to show that the \eqref{sample_final} has a unique solution. However, as discussed in the paper, \eqref{sample_final} is simply a reformulation of \eqref{cvar_emp}. Thus, we can equivalently show that the solution to \eqref{cvar_emp} is unique. 

In \textit{Theorem \ref{lem3}}, we have shown that the term $\bar{q}_{i_j}(K)$ is strongly convex function in $K$. The uniqueness of the optimal solution, $K_N^*$, then follows from the strict convexity of the objective function. \\
\textbf{Proof that $(b)$ holds:} We prove by contradiction. Let us assume that $\exists K, l$ and independent samples $(\delta_1,\dots,\delta_N)$ such that $\textbf{P}^N\{\bar{q}_{i}(K)=l,i=1,2,\dots,N\} = \theta \geq 0$. Then by definition, $\exists \;\delta_j$ such that $\textbf{P}\{ \delta_j | \bar{q}_{j}(K)=l\} \neq 0$. However this contradicts the assumption that the distribution of $\bar{q}(\cdot)$ has no point masses. This concludes the proof.
\end{proof}}

\bibliographystyle{ieeetr}
\bibliography{ms}

\end{document}